\documentclass[12pt]{elsarticle}

\usepackage{amsmath}
\usepackage{amssymb}
\usepackage{hyperref}
\usepackage{eucal}
\usepackage{enumerate}

\newtheorem{theorem}{Theorem}[section]

\newtheorem{proposition}[theorem]{Proposition}
\newtheorem{question}{Question}

\newdefinition{remark}[theorem]{Remark}
\newdefinition{definition}[theorem]{Definition}
\newdefinition{example}[theorem]{Example}

\newproof{proof}{Proof}
\newproof{poThmOne}{Proof of Theorem \ref{T:intro}}

\numberwithin{equation}{section}

\newcommand{\B}{\mathbb{B}}
\newcommand{\C}{\mathbb{C}}

\newcommand{\mcM}{\mathcal{M}}
\newcommand{\mcH}{\mathcal{H}}

\newcommand{\bfx}{\mathbf{x}}
\newcommand{\bfy}{\mathbf{y}}

\newcommand{\bfT}{\mathbf{T}}
\newcommand{\bfV}{\mathbf{V}}
\newcommand{\bfS}{\mathbf{S}}
\newcommand{\bfU}{\mathbf{U}}
\newcommand{\bfN}{\mathbf{N}}

\newcommand{\bfW}{\mathbf{W}}
\newcommand{\bfX}{\mathbf{X}}
\newcommand{\bfY}{\mathbf{Y}}

\newcommand{\bfI}{\mathbf{I}}

\newcommand{\calH}{\mathcal{H}}
\newcommand{\calK}{\mathcal{K}}

\newcommand{\calV}{\mathcal{V}}
\newcommand{\calJ}{\mathcal{J}}

\newcommand{\inner}[2]{{\langle #1,#2\rangle}}
\newcommand{\binner}[2]{{\Big\langle #1,#2\Big\rangle}}

\allowdisplaybreaks[1]

\journal{Journal of Functional Analysis}

\date{June 8, 2019}

\begin{document}

\begin{frontmatter}



\title{Decomposing algebraic $m$-isometric tuples}


\author{Trieu Le}

\address{Department of Mathematics and Statistics, The University of
Toledo, Toledo, OH 43606}

\ead{trieu.le2@utoledo.edu}

\begin{abstract}
We show that any $m$-isometric tuples of commuting operators on a finite dimensional Hilbert space can be decomposed as a sum of a spherical isometry and a commuting nilpotent tuple. Our approach applies as well to tuples of algebraic operators that are hereditary roots of polynomials in several variables.
\end{abstract}

\begin{keyword}
$m$-isometry \sep nilpotent \sep commuting tuple


\MSC[2010] 47A05 \sep 47A13

\end{keyword}

\end{frontmatter}

\section{Introduction}
\label{S:intro}
The notion of $m$-isometries was introduced and studied by Agler \cite{AglerAJM1990} back in the eighties. A bounded linear operator $T$ on a complex Hilbert space $\mcH$ is called $m$-isometric if it satisfies the operator equation
\begin{align*}
\sum_{k=0}^{m}(-1)^{m-k}\binom{m}{k}T^{*k}T^{k}=0,
\end{align*}
where $T^{*}$ is the adjoint operator of $T$. Equivalently, for all $v\in\mcH$,
\[\sum_{k=0}^{m}(-1)^{m-k}\binom{m}{k}\|T^{k}v\|^2=0.\]
In a series of papers \cite{AglerIEOT1995,AglerIEOT1995b,AglerIEOT1996}, Agler and Stankus gave an extensive study of $m$-isometric operators. It is clear that any $1$-isometric operator is an isometry. Multiplication by $z$ on the Dirichlet space over the unit disk is not an isometry but it is a $2$-isometry. Richter \cite{RichterTAMS1991} showed that any cyclic $2$-isometry arises from multiplication by $z$ on certain Dirichlet-type spaces. Very recently, researchers have been interested in algebraic properties, cyclicity and supercyclicity of $m$-isometries, among other things. See \cite{PatelGMS2002,JablonskiIEOT2002,BotelhoSzeged2010,DuggalLAA2012,MuneoPAMS2013,Faghih-AhmadiRMJM2013,BermudezLAA2013,BermudezJMAA2013,LeJMAA2014,BermudezAAA2014,GuLAA2015} and the references therein.

It was showed by Agler, Helton and Stankus \cite[Section 1.4]{AglerLAA1998} that any $m$-isometry $T$ on a finite dimensional Hilbert space admits a decomposition $T=S+N$, where $S$ is a unitary and $N$ is a nilpotent operator satisfying $SN=NS$. In \cite{BermudezJMAA2013}, it was showed that if $S$ is an isometry on any Hilbert space and $N$ is a nilpotent operator of order $n$ commuting with $S$ then the sum $S+N$ is a strict $(2n-1)$-isometry. This result has been generalized to $m$-isometries by several authors \cite{LeJMAA2014,BermudezAAA2014,GuLAA2015}.

Let $A$ be a positive operator on $\mcH$. An operator $T$ is called an $(A,m)$-isometry if it is a solution to the operator equation
\begin{align*}
\sum_{k=0}^{m}(-1)^{m-k}\binom{m}{k}T^{*k}AT^{k}=0.
\end{align*}
Such operators were introduced and studied by Sid Ahmed and Saddi in \cite{AhmedLAA2012}, then by other authors \cite{DuggalFAAC2012,JungStudia2012,RabaouiAMS2012,HedayatianHMJ2015,Faghih-AhmadiBMMSS2016,BermudezLAA2018}. In the case $m=1$, we call such operators $A$-isometries. Since $A$ is positive, the map $v\mapsto \|v\|_{A} := \inner{Av}{v}$ (where $\inner{\cdot}{\cdot}$ denotes the inner product on $\mcH$) gives rise to a seminorm. In the case $A$ is injective, $\|\cdot\|_{A}$ becomes a norm. It follows that an operator $T$ is $(A,m)$-isometric if and only if $T$ is $m$-isometric with respect to $\|\cdot\|_{A}$. As a result, several algebraic properties of $(A,m)$-isometries follow from the corresponding properties of $m$-isometries with more or less similar proofs (see \cite{AhmedLAA2012, BermudezLAA2018}). However, there are great differences between $(A,m)$-isometries and $m$-isometries, specially when $A$ is not injective. For example, it is known \cite{AglerIEOT1995} that the spectrum of an $m$-isometry must either be a subset of the unit circle or the entire closed unit disk. On the other hand, \cite[Theorem~2.3]{BermudezLAA2018} shows that for any compact set $K$ on the plane that intersects the unit circle, there exist a non-zero positive operator $A$ and an $(A,1)$-isometry whose spectrum is exactly $K$. The following question was asked in \cite{BermudezLAA2018}.

\begin{question} Let $T$ be an $(A,m)$-isometry on a finite dimensional Hilbert space. Is it possible to write $T$ as a sum of an $A$-isometry and a commuting nilpotent operator?
\end{question}
In this paper, we shall answer Question 1 in the affirmative. Indeed, we are able to prove a much more general result, in the setting of multivariable operator theory.

Gleason and Richter \cite{GleasonIEOT2006} considered the multivariable setting of $m$-isometries and studied their properties. A commuting $d$-tuple of operators $\bfT=[T_1,\ldots, T_d]$ is said to be an $m$-isometry if it satisfies the operator equation
\begin{align}
\label{L:m_isometric_tuples}
\sum_{k=0}^{m}(-1)^{m-k}\binom{m}{k}\sum_{|\alpha|=k}\frac{k!}{\alpha!}\,(T^{\alpha})^{*}T^{\alpha}=0.
\end{align}
Here $\alpha=(\alpha_1,\ldots,\alpha_d)$ denotes a multiindex of non-negative integers. We have also used the standard multiindex notation: $|\alpha|=\alpha_1+\cdots+\alpha_d$, $\alpha!=\alpha_1!\cdots\alpha_d!$ and $\bfT^{\alpha}=T_1^{\alpha_1}\cdots T_d^{\alpha_{d}}$. It was shown in \cite{GleasonIEOT2006} that the $d$-shift on the Drury-Arveson space over the unit ball in $\C^d$ is $d$-isometric. This generalizes the single-variable fact that the unilateral shift on the Hardy space $H^2$ over the unit disk is an isometry. Gleason and Richter also studied spectral properties of $m$-isometric tuples and they constructed a list of examples of such operators, built from single-variable $m$-isometries. Many algebraic properties of $m$-isometric tuples have been discovered by the author in an unpublished work and independently by Gu \cite{GuJKMS2018}. As an application of our main result in this note, we shall answer the following question in the affirmative.

\begin{question}
Let $\bfT$ be an $m$-isometric tuple acting on a finite dimensional Hilbert space. Is it possible to write $\bfT$ as a sum of a $1$-isometric $\bfS$ (that is, a spherical isometry) and a nilpotent tuple $\bfN$ that commutes with $\bfS$?
\end{question}

To state our main result, we first generalize the notion of $(A,m)$-isometric operators to tuples. Let $A$ be any bounded operator on $\mcH$ (we do not need to assume that $A$ is positive). A commuting tuple $\bfT=[T_1, \ldots, T_d]$ is said to be $(A,m)$-isometric if
\begin{align}
\label{L:A_m_isometric_tuples}
\sum_{k=0}^{m}(-1)^{m-k}\binom{m}{k}\sum_{|\alpha|=k}\frac{k!}{\alpha!}\,(T^{\alpha})^{*}AT^{\alpha}=0.
\end{align}
It is clear that $(I,m)$-isometric tuples (here $I$ stands for the identity operator) are the same as $m$-isometric tuples. We shall call $(A,1)$-isometric tuples \textit{spherical $A$-isometric}. They are tuples $\bfT$ that satisfies
\[T_1^{*}AT_1 + \cdots + T_d^{*}AT_d = A.\]

A main result in the paper is the following theorem.

\begin{theorem}
\label{T:intro}
Suppose $\bfT$ is an $(A,m)$-isometric tuple on a finite dimensional Hilbert space. Then there exist a  spherical $A$-isometric tuple $\bfS$ and a nilpotent tuple $\bfN$ commuting with $\bfS$ such that $\bfT=\bfS+\bfN$.
\end{theorem}

In the case of a single operator, Theorem \ref{T:intro} answers Question 1 in the affirmative. In the case $A=I$, we also obtain an affirmative answer to Question 2.

\section{Hereditary calculus and applications}

Our approach uses a generalization of the hereditary functional calculus developed by Agler \cite{Agler1980,AglerJOT1982}. We begin with some definitions and notation. We use boldface lowercase letters, for example $\bfx$, $\bfy$, to denote $d$-tuples of complex variables. Let $\C[\bfx,\bfy]$ denote the space of polynomials in commuting variables $\bfx$ and $\bfy$ with complex coefficients.
Let $A$ be a bounded linear operator on a Hilbert space $\mcH$ and $\bfX$, $\bfY$ be two $d$-tuples of commuting bounded operators on $\mcH$. These two tuples may not commute with each other. We denote by $\bfX^{*}$ the tuple $[X_1^{*},\ldots, X_d^{*}]$. Let $f\in\C[\bfx,\bfy]$. If
\[f(\bfx,\bfy) = \sum_{\alpha,\beta}c_{\alpha,\beta}\bfx^{\alpha}\bfy^{\beta},\] where the sum is finite, then we define
\begin{align}
\label{Eqn:A_hereditaryDef}
f(A;\bfX,\bfY) = \sum_{\alpha,\beta}c_{\alpha,\beta}(\bfX^{\alpha})^{*}A\bfY^{\beta}.
\end{align}
It is clear that the map $f\mapsto f(A;\bfX,\bfY)$ is linear from $\C[\bfx,\bfy]$ into $(\mathcal{B}(\mcH))^{d}$. If $g\in\C[\bfx,\bfy]$ depending only on $\bfx$, then $g(A;\bfX,\bfY) = g(\bfX^{*})A$. On the other hand, if $h\in\C[\bfx,\bfy]$ depending only on $\bfy$, then $h(A;\bfX,\bfY)=A\,h(\bfY)$. Furthermore, if $F=g\, f\, h$, then
\begin{align}
\label{Eqn:multiplicative_prop}
F(A;\bfX,\bfY) = g(X^{*})f(A;\bfX,\bfY)h(\bfY).
\end{align}
If $\bfX=\bfY$, we shall write $f(A;\bfX)$ instead of $f(A;\bfX,\bfX)$. In the case $A=I$, the identity operator, we shall use $f(\bfX,\bfY)$ to denote $f(I;\bfX,\bfY)$. Therefore, $f(\bfX)$ denotes $f(I;\bfX,\bfX)$. We say that $\bfX$ is a \textit{hereditary root} of $f$ if $f(\bfX)=0$.

\begin{example}
\label{E:m-isometries}
Define $p_m(\bfx,\bfy) = \big(\sum_{j=1}^{d} x_jy_j-1)^{m}\in\C[\bfx,\bfy]$. It is then clear that $\bfT$ is $m$-isometric if and only if $\bfT$ is a hereditary root of $p_m$, that is, $p_m(\bfT)=0$. Similarly, $\bfT$ is $(A,m)$-isometric if and only if $p_m(A;\bfT)=0$.
\end{example}

Even though the map $f\mapsto f(A;\bfX,\bfY)$ is not multiplicative in general, it turns out that its kernel is an ideal of $\C[\bfx,\bfy]$. This observation will play an important role in our approach.

\begin{proposition}
\label{P:A_ideal_property}
Let $A$ be a bounded linear operator and let $\bfX$ and $\bfY$ be two $d$-tuples of commuting operators. Define
$$\mathcal{J}(A;\bfX,\bfY) = \big\{f\in\C[\bfx,\bfy]: f(A;\bfX,\bfY) = 0\big\}.$$
Then $\mathcal{J}(A;\bfX,\bfY)$ is an ideal of $\C[\bfx,\bfy]$.
\end{proposition}

\begin{proof}
For simplicity of the notation, throughout the proof, let us write $\mathcal{J}$ for $\mathcal{J}(A;\bfX,\bfY)$. It is clear that $\mathcal{J}$ is a vector subspace of $\C[\bfx,\bfy]$. Now let $f$ be in $\mathcal{J}$ and $g$ be in $\C[\bfx,\bfy]$. We need to show that $gf$ belongs to $\mathcal{J}$. By linearity, it suffices to consider the case $g$ is a monomial $g(\bfx,\bfy) = \bfx^{\alpha}\bfy^{\beta}$ for some multi-indices $\alpha$ and $\beta$. By \ref{Eqn:multiplicative_prop},
\begin{align*}
(fg)(A;\bfX,\bfY) & = (\bfX^{\alpha})^{*}\cdot f(A;\bfX,\bfY)\cdot \bfY^{\beta} = 0,
\end{align*}
since $f(A;\bfX,\bfY)=0$. This shows that $fg$ belongs to $\mathcal{J}$ as desired.
\end{proof}

If $f$ is a polynomial of $\bfy$ in the form $f(\bfy)=\sum_{\alpha}c_{\alpha}\bfy^{\alpha}$, we define $\bar{f}(\bfx)$ as
\[\bar{f}(\bfx) = \sum_{\alpha}\bar{c}_{\alpha}\bfx^{\alpha}.\] In the case $A$ is positive and $\bfX=\bfY$, we obtain an additional property of the ideal $\calJ(A;\bfY,\bfY)$ as follow.

\begin{proposition}
\label{P:A_positive_hereditary}
Let $A$ be a positive operator and $\bfY$ be a $d$-tuple of commuting operators. Suppose $f_1,\ldots, f_m$ are polynomials of $\bfy$ such that the sum $\bar{f}_1(\bfx)f_1(\bfy)+\cdots+\bar{f}_m(\bfx)f_m(\bfy)$ belongs to $\calJ(A;\bfY,\bfY)$. Then $f_1(\bfy),\ldots,f_m(\bfy)$ also belong to $\calJ(A;\bfY,\bfY)$.
\end{proposition}

\begin{proof}
Note that $\bar{f}_j(\bfY^{*}) = (f_j(\bfY))^{*}$ for all $j$. By the hypotheses, we have
\[(f_1(\bfY))^{*}Af_1(\bfY) + \cdots + (f_m(\bfY))^{*}Af_m(\bfY)=0,\]
which implies
\[[(A^{1/2}f_1(\bfY)]^{*}[A^{1/2}f_1(\bfY)] + \cdots + [A^{1/2}f_m(\bfY)]^{*}[A^{1/2}f_m(\bfY)]=0.\]
It follows that for all $j$, we have $A^{1/2}f_j(\bfY)=0$, which implies $Af_j(\bfY)=0$. Therefore, $f_j(\bfy)\in\calJ(A;\bfY,\bfY)$ for all $j$.
\end{proof}

Recall that the radical ideal of an ideal $\mathcal{J}\subset\C[\bfx,\bfy]$, denoted by ${\rm Rad}(\mathcal{J})$, is the set of all polynomials $p\in\C[\bfx,\bfy]$ such that $p^{N}\in\mathcal{J}$ for some positive integer $N$. In the following proposition, we provide an interesting relation between generalized eigenvectors and eigenvalues of $\bfX$ and $\bfY$ whenever we have $f(A;\bfX,\bfY)=0$.

\begin{proposition}
\label{P:generalized_eigenvectors}
Let $\bfX$ and $\bfY$ be two $d$-tuples of commuting operators. Suppose $k$ is a positive integer, $\lambda=(\lambda_1,\ldots,\lambda_d), \omega=(\omega_1,\ldots,\omega_d)\in\C^d$ and $u,v\in\calH$ such that 
\[(X_j - {\lambda}_j)^{k}u=(Y_j-\omega_j)^{k}v=0\] for all $1\leq j\leq d$. Then for any polynomial $f\in{\rm Rad}(\calJ(A;\bfX,\bfY))$, we have
\begin{align}
\label{Eqn:vanishing_ideal}
f(\bar{\lambda},\omega)\inner{Av}{u}=0.
\end{align}
\end{proposition}

\begin{proof}
We first assume that $f\in\calJ(A;\bfX,\bfY)$. Using Taylor's expansion, we find polynomials $g_1,\ldots, g_d$ and $h_1,\ldots,h_d$ such that
\[f(\bar{\lambda},{\omega})-f(\bfx,\bfy) = \sum_{j=1}^{d}(x_j-\bar{\lambda}_j)g_j(\bfx,\bfy) + \sum_{j=1}^{d}h_j(\bfx,\bfy)(y_j-{\omega}_j).\]
Take any integer $M\geq 1+2d(k-1)$. By the multinomial expansion, there exist polynomials $G_1,\ldots, G_d$ and $H_1,\ldots,H_d$ such that
\begin{align*}
\Big(f(\bar{\lambda},{\omega})-f(\bfx,\bfy))\Big)^{M} = \sum_{j=1}^{d}(x_j-\bar{\lambda}_j)^{k}G_j(\bfx,\bfy) + \sum_{j=1}^{d}H_j(\bfx,\bfy)(y_j-{\omega}_j)^{k}.
\end{align*}
The left-hand side, by the binomial expansion, can be written as
\[ (f(\bar{\lambda},{\omega}))^{M}+f(\bfx,\bfy)H(\bfx,\bfy)\] for some polynomial $H$. Since $f(A;\bfX,\bfY)=0$, using Equation \eqref{Eqn:multiplicative_prop} and Proposition \ref{P:A_ideal_property}, we conclude that
\[(f(\bar{\lambda},{\omega}))^M\cdot A= \sum_{j=1}^{d}(X_j^{*} - \bar{\lambda}_j)^{k}G_j(A;\bfX,\bfY) + \sum_{j=1}^{d}H_j(A;\bfX,\bfY)(Y_j-{\omega}_j)^{k}.\]
Consequently,
\begin{align*}
& (f(\bar{\lambda},{\omega}))^M\inner{Av}{u} \\
& \ = \sum_{j=1}^{d}\binner{G_j(A;\bfX,\bfY)v}{(X_j-{\lambda}_j)^{k}u} + \sum_{j=1}^{d}\binner{H_j(A;\bfX,\bfY)(Y_j-\omega_j)^{k}v}{u}\\
& \ = 0,
\end{align*}
which implies \eqref{Eqn:vanishing_ideal}. 

In the general case, there exists an integer $N\geq 1$ such that $f^N$ belongs to $\calJ(A;\bfX,\bfY)$. By the case we have just proved, $(f(\bar{\lambda},\omega))^{N}\inner{Av}{u}=0$, which again implies \eqref{Eqn:vanishing_ideal}. This completes the proof of the proposition.
\end{proof}

\begin{remark} In the case of a single operator, Proposition \ref{P:generalized_eigenvectors} provides a generalization of \cite[Lemmas 18 and 19]{AglerLAA1998}. Our proof here is even simpler and more transparent.
\end{remark}

Question 1 and Question 2 in the introduction concern operators acting on a finite dimensional Hilbert space. It turns out that this condition can be replaced by a weaker one. Recall that a linear operator $T$ is called \textit{algebraic} if there exists complex constants $c_0, c_1,\ldots, c_{\ell}$ such that
\[c_0 I + c_1 T + \cdots + c_{\ell}T^{\ell} = 0.\]
Algebraic operator roots of polynomials were investigated in \cite{AglerLAA1998}.

We first discuss some preparatory results on algebraic operators acting on a  general complex vector space $\calV$. It is well known that if $T$ is an algebraic linear operator on $\calV$, then the spectrum $\sigma(T)$ is finite and there exists a direct sum decomposition $\calV = \oplus_{a\in\sigma(T)}\calV_{a}$, where each $\calV_{a}$ is an invariant subspace for $T$ (the subspace $\calV_{a}$ is a closed subspace if $\calV$ is a normed space and $T$ is bounded) and $T-a I$ is nilpotent on $\calV_{a}$. Indeed, if the minimal polynomial of $T$ is factored in the form
\[p(z) = (z-a_1)^{m_1}\cdots (z-a_{\ell})^{m_{\ell}},\] where $a_1,\ldots,a_{\ell}$ are pairwise distinct and $m_1,\ldots,m_{\ell}\geq 1$, then $\sigma(T)=\{a_1,\ldots,a_{\ell}\}$ and $\calV_{a_j}=\ker(T-a_j)^{m_j}$ for $1\leq j\leq \ell$. See, for example, \cite[Section 6.3]{ShilovDoverPub1977}, which discusses operators acting on finite dimensional vector spaces. However, the arguments apply to algebraic operators on infinite dimensional vector spaces as well.

Suppose now $\bfT=[T_1,\ldots, T_d]$ is a tuple of commuting algebraic operators on $\calV$. We first decompose $\calV$ as above with respect to the spectrum $\sigma(T_1)$. Since each subspace in the decomposition is invariant for all $T_j$, we again decompose such subspace with respect to the spectrum $\sigma(T_2)$. Continuing this process, we obtain a finite set $\Lambda\subset\C^d$ and a direct sum decomposition $\calV=\oplus_{\lambda\in\Lambda}\calV_{\lambda}$ such that for each $\lambda=(\lambda_1,\ldots,\lambda_d)\in\Lambda$ and $1\leq j\leq d$, the subspace $\calV_{\lambda}$ is invariant for $\bfT$ and $T_j-\lambda_j I$ is nilpotent on $\calV_{\lambda}$. Let $E_{\lambda}$ denote the canonical projection (possibly non-orthogonal) from $\calV$ onto $\calV_{\lambda}$. Then we have $\sum_{\lambda\in\Lambda}E_{\lambda}=I$, $E_{\lambda}^2=E_{\lambda}$, and $E_{\lambda}E_{\gamma}=0$ if $\lambda\neq\gamma$. Define
\begin{align}
\label{Eqn:Jordan_decomposition}
\bfS = \sum_{\lambda\in\Lambda}\lambda\cdot E_{\lambda} = \Big[\sum_{\lambda\in\Lambda}{\lambda}_1 E_{\lambda},\ldots,\sum_{\lambda\in\Lambda}{\lambda}_d E_{\lambda}\Big]
\end{align}
Then $\bfS$ is a tuple of commuting operators which commutes with $\bfT$, and $\bfT-\bfS$ is nilpotent. For any multiindex $\alpha$, we have
\[\bfS^{\alpha} = S_1^{\alpha_1}\cdots S_d^{\alpha_d} = \sum_{\lambda\in\Lambda}\lambda^{\alpha}E_{\lambda}.\]
In the case $\calV$ is a normed space and $\bfT$ is bounded, each operator in the tuple $\bfS$ is bounded as well. 

We now prove a very general result, which will provide affirmative answers to Questions 1 and 2 in the introduction.

\begin{theorem}
\label{T:decomposition_tuples}
Let $\bfX$ and $\bfY$ be two $d$-tuples of commuting algebraic operators on a Hilbert space $\calH$.  Let $\bfU$ (respectively, $\bfV$) be the commuting tuple associated with $\bfX$ (respectively, $\bfY$) as in \eqref{Eqn:Jordan_decomposition}. Then
\begin{align}
\label{Eqn:decomposition_tuples}
{\rm Rad}(\calJ(A;\bfX,\bfY)) \subseteq \calJ(A;\bfU,\bfV).
\end{align}
\end{theorem}

\begin{proof}
Write $\bfX=[X_1, \ldots, X_d]$ and decompose $\calH=\oplus_{\lambda\in\Lambda}\calH_{\lambda}$ such that for each $\lambda=(\lambda_1,\ldots,\lambda_d)\in\Lambda$, the subspace $\calH_{\lambda}$ is invariant for $\bfX$ and $X_j - {\lambda}_j I$ is nilpotent on $\calH_{\lambda}$. Let $U_{\lambda}$ denote the canonical projection from $\calH$ onto $\calH_{\lambda}$. Then
$\bfU = \sum_{\lambda\in\Lambda}{\lambda}\cdot U_{\lambda}$ and for any multiindex $\alpha$, we have
\[\bfU^{\alpha} = \sum_{\lambda\in\Lambda}{\lambda}^{\alpha}\cdot U_{\lambda}.\]
Similarly, write $\bfY=[Y_1,\ldots, Y_d]$ and decompose $\calH=\oplus_{\omega\in\Omega}\calK_{\omega}$. Let $V_{\omega}$ be the canonical projection from $\calH$ onto $\calK_{\omega}$. Then
$\bfV = \sum_{\omega\in\Omega}\omega\cdot V_{\omega}$ and for any multiindex $\beta$, 
\[\bfV^{\beta} = \sum_{\omega\in\Omega}\omega^{\beta}\cdot V_{\omega}.\]
Take any polynomial $p\in{\rm Rad}(\calJ(A;\bfX,\bfY))$. For $\lambda\in\Lambda$, $\omega\in\Omega$ and vectors $u\in\calH_{\lambda}$ and $v\in\calK_{\omega}$, there exists an integer $k\geq 1$ sufficiently large such that \[(X_j-\lambda_j I)^{k}u=(Y_j-\omega_j I)^{k}v=0\] for all $1\leq j\leq d$. Proposition \ref{P:generalized_eigenvectors} shows that $p(\bar{\lambda},\omega)\inner{Av}{u}=0$,
which implies \[p(\bar{\lambda},\omega)U_{\lambda}^{*}AV_{\omega} = 0.\]
Write $p(\bfx,\bfy) = \sum_{\alpha,\beta} c_{\alpha,\beta}\bfx^{\alpha}\bfy^{\beta}$. We compute
\begin{align*}
p(A;\bfU,\bfV) & = \sum_{\alpha,\beta}c_{\alpha,\beta}\bfU^{*\alpha}A\bfV^{\beta}\\
& = \sum_{\alpha,\beta}c_{\alpha,\beta}\Big(\sum_{\lambda\in\Lambda}\bar{\lambda}^{\alpha}U_{\lambda}^{*}\Big)A\Big(\sum_{\omega\in\Omega}\omega^{\beta}\cdot V_{\omega}\Big)\\
& = \sum_{\lambda\in\Lambda, \omega\in\Omega}\Big(\sum_{\alpha,\beta}c_{\alpha,\beta}\bar{\lambda}^{\alpha}\omega^{\beta}\Big)U_{\lambda}^{*}AV_{\omega}\\
& = \sum_{\lambda\in\Lambda, \omega\in\Omega}p(\bar{\lambda},\omega)U_{\lambda}^{*}AV_{\omega} = 0.
\end{align*}
We conclude that $p\in\calJ(A;\bfU,\bfV)$. Since $p\in{\rm Rad}(\calJ(A;\bfX,\bfY))$ was arbitrary, the proof of the theorem is complete.
\end{proof}

Theorem \ref{T:decomposition_tuples} enjoys numerous interesting applications that we now describe.

\begin{poThmOne} We shall prove the theorem under a more general assumption that $\bfT$ is a tuple of commuting algebraic operators. Since $\bfT$ is $(A,m)$-isometric, the polynomial $(\sum_{j=1}^{d}x_jy_j-1)^m$ belongs to the ideal $\calJ(A;\bfT,\bfT)$. It follows that the polynomial $p(x,y)=\sum_{j=1}^{d}x_jy_j -1$ belongs to the radical ideal of $\calJ(A;\bfT,\bfT)$. By Theorem \ref{T:decomposition_tuples}, we may decompose $\bfT=\bfS+\bfN$, where $\bfN$ is a nilpotent tuple commuting with $\bfS$ and $p(A;\bfS,\bfS)=0$, which means that $\bfS$ is a spherical $A$-isometry. 
\end{poThmOne}

\begin{example}
Recall that an operator $T$ is called $(m,n)$-isosymmetric (see \cite{StankusIEOT2013}) if $T$ is a hereditary root of $f(x,y)=(xy-1)^m\,(x-y)^n$. Theorem \ref{T:decomposition_tuples} shows that any such algebraic $T$ can be decomposed as $T=S+N$, where $N$ is nilpotent, $S$ is isosymmetric (i.e. $(1,1)$-isosymmetric) and $SN=NS$.
\end{example}

\begin{example} Several researchers \cite{MohammadiRMJM2019, AthavaleP1999} have investigated the so-called toral $m$-isometric tuples. It is straightforward to generalize this notion to toral $(A,m)$-isometric tuples, which are commuting $d$-tuples $\bfT$ that satisfy
\[\sum_{\substack{0\leq \alpha_1\leq m_1\\ \ldots\\ 0\leq \alpha_d\leq m_d}} (-1)^{|\alpha|}\binom{m_1,\ldots, m_d}{\alpha}(\bfT^{\alpha})^{*}\,A\,\bfT^{\alpha} = 0.\] for all $m_1+\cdots+m_d=m$. Equivalently, $\bfT$ is a common hereditary root of all polynomials of the form $(1-x_1y_1)^{m_1}\cdots(1-x_dy_d)^{m_d}$ for $m_1+\cdots+m_d=m$. This means that all these polynomials belong to the ideal $\calJ(A;\bfT,\bfT)$. We see that toral $(A,1)$-isometries are just commuting tuples $\bfT$ such that each $T_j$ is an $A$-isometry, that is, $T_j^{*}AT_j=A$. Note that for any toral $(A,m)$-isometry $\bfT$, the radical ideal ${\rm Rad}(\calJ(A;\bfT,\bfT))$ contains all the polynomials $\big\{1-x_jy_j: j=1,2,\ldots, d\big\}$. Theorem \ref{T:decomposition_tuples} asserts that $\bfT=\bfS+\bfN$, where $\bfS$ is a toral $(A,1)$-isometry and $\bfN$ is a nilpotent tuple commuting with $\bfS$. 
\end{example}

\section{On $2$-isometric tuples}
It is well known that any $2$-isometry on a finite dimensional Hilbert space must actually be an isometry. On the other hand, there are many examples of finite dimensional $2$-isometric tuples that are not spherical isometries. The following class of examples is given in Richter's talk \cite{RichterCornell2011}.

\begin{example} 
\label{Ex:two_iso_tuples}
If $\alpha=(\alpha_1,\ldots,\alpha_d)\in\partial\B_d$ and $V_j:\C^m\rightarrow\C^n$ such that $\sum_{j=1}^{d}\overline{\alpha}_jV_j=0$, then $\bfW=(W_1,\ldots, W_d)$ with
\[W_j = \begin{pmatrix}
\alpha_j I_{n} & V_j\\
0 & \alpha_j I_{m}
\end{pmatrix}
\] defines a $2$-isometric $d$-tuple.
\end{example} 

The following result was stated in \cite{RichterCornell2011} without a proof and as far as the author is aware of, it has not appeared in a published paper. 

\begin{theorem}[Richter-Sundberg]
\label{T:RS}
If $\bfT$ is a $2$-isometric tuple on a finite dimensional space, then
\[\bfT = \bfU\oplus \bfW,\]
where $\bfU$ is a spherical unitary and $\bfW$ is a direct sum of operator tuples unitarily equivalent to those in Example \ref{Ex:two_iso_tuples}.
\end{theorem}

In this section, we shall assume that $A$ is self-adjoint and investigate $(A,2)$-isometric $d$-tuples. We obtain a characterization for such tuples that generalizes the above theorem. We first provide a generalization of Example \ref{Ex:two_iso_tuples}. We call $\bfN=(N_1,\ldots,N_d)$ an $(A,n)$-nilpotent tuple if $A\bfN^{\alpha} = 0$ for any indices $\alpha$ with $|\alpha|=n$.

\begin{proposition}
\label{P:A_two_iso_tuples}
Assume that $A$ is a self-adjoint operator. Let $\bfS$ be an $(A,1)$-isometry and $\bfN$ an $(A,2)$-nilpotent tuple such that $\bfS$ commutes with $\bfN$. Suppose $S_1^{*}AN_1+\cdots+S_d^{*}AN_d=0$, then $\bfS+\bfN$ is an $(A,2)$-isometry.
\end{proposition}

\begin{proof}
By the assumption, we have $AN_jN_k=N_j^{*}N_k^{*}A=0$ for $1\leq j,k\leq d$, $\sum_{j=1}^{d}S_j^{*}AS_j = A$, and $\sum_{j=1}^{d}S_j^{*}AN_j = \sum_{j=1}^{d}N_j^{*}AS_j = 0$. It follows that
\begin{align*}
\sum_{j=1}^{d}(S_j+N_j)^{*}A\,(S_j+N_j) &  = A + \sum_{j=1}^{d}N_j^{*}AN_j.
\end{align*}
We then compute
\begin{align*}
&\sum_{1\leq k,j\leq d}(S_k+N_k)^{*}(S_j+N_j)^{*}A\,(S_j+N_j)(S_k+N_k) \\
& \qquad = \sum_{k=1}^{d}(S_k^{*}+N_k^{*})(A + \sum_{j=1}^{d} N_j^{*}AN_j)(S_k+N_k)\\
& \qquad = \sum_{k=1}^{d}(S_k^{*}+N_k^{*})A\,(S_k+N_k) + \sum_{1\leq k, j\leq d}(S_k^{*}+N_k^{*})N_j^{*}AN_j(S_k+N_k)\\
& \qquad = A + \sum_{k=1}^{d}N_k^{*}AN_k + \sum_{1\leq k,j\leq d}S_k^{*}N_j^{*}AN_jS_k\\
& \qquad = A + \sum_{k=1}^{d}N_k^{*}AN_k + \sum_{j=1}^{d}N_j^{*}\Big(\sum_{k=1}^{d} S_k^{*}AS_k\Big)N_j \\
& \qquad = A + \sum_{k=1}^{d}N_k^{*}AN_k + \sum_{j=1}^{d}N_j^{*}AN_j\\
& \qquad = A + 2\sum_{j=1}^{d}N_j^{*}AN_j\\
& \qquad = 2\sum_{j=1}^{d}(S_j+N_j)^{*}A\,(S_j+N_j) - A.
\end{align*}
Consequently, the sum $\bfS+\bfN$ is an $(A,2)$-isometric tuple.
\end{proof}

\begin{remark} We have provided a direct proof of Proposition \ref{P:A_positive_hereditary}. Using the hereditary functional calculus and the approach in \cite{LeJMAA2014}, one may generalize the result to the case $\bfS$ being an $(A,m)$-isometry and $\bfN$ an $(A,n)$-nilpotent commuting with $\bfS$. Under such an assumption, if $S_1^{*}AN_1+\cdots+S_d^{*}AN_d=0$, then $\bfS+\bfN$ is an $(A,m+2n-3)$-isometry. We leave the details for the interested reader.
\end{remark}

We now show that any algebraic $(A,2)$-isometric tuple has the form given in Proposition \ref{P:A_two_iso_tuples} and as a result, provide a proof of Richter-Sundberg's theorem.

\begin{theorem}
\label{T:A_2_isometric_tuples}
Assume that $A$ is a positive operator. Let $\bfT$ be an algebraic $(A,2)$-isometric tuple on $\mcH$. Then there exists an $(A,1)$-isometric tuple $\bfS$ and a tuple $\bfN$ commuting with $\bfS$ such that $\bfT = \bfS+\bfN$, $\sum_{\ell=1}^{d}S_{\ell}^{*}AN_{\ell}=0$, and $AN_{j}N_{\ell}=0$ for all $1\leq j,\ell\leq d$ (we call such $N$ an $(A,2)$-nilpotent tuple).

In the case $\mcH$ is finite dimensional and $A=I$, the identity operator, we recover Theorem \ref{T:RS}.
\end{theorem}

\begin{proof}
Recall that there exists a finite set $\Lambda\subset\C^d$ and a direct sum decomposition $\mcH=\oplus_{\lambda\in\Lambda}\mcH_{\lambda}$ such that for each $\lambda\in\Lambda$, the subspace $\mcH_{\lambda}$ is invariant for $\bfT$ and $T_j-\lambda_j I$ is nilpotent on $\mcH_{\lambda}$. Let $\bfS$ be defined as in \eqref{Eqn:Jordan_decomposition} and put $\bfN = \bfT-\bfS$. From the construction, $\bfN$ is nilpotent and Theorem \ref{T:decomposition_tuples} shows that $\bfS$ is $(A,1)$-isometric. We shall show that $\bfN$ satisfies the required properties.

Restricting on each invariant subspace $\mcH_{\lambda}$, we only need to consider the case $\mcH=\mcH_{\lambda}$ and so $\bfS=\lambda\mathbf{I}$. Proposition \ref{P:generalized_eigenvectors} asserts that $(|\lambda|^2-1)\inner{Av}{u}=0$ for all $v,u\in\mcH$. If $|\lambda|\neq 1$, then $A=0$ and the conclusion follows. Now we assume that $|\lambda|=1$. Since $\bfN$ is nilpotent, there exists a positive integer $r$ such that $A\,\bfN^{\alpha}=0$ whenever $|\alpha|=r$. We claim that $r$ may be taken to be $2$. To prove the claim, we assume $r\geq 3$ and show that $A\,\bfN^{\alpha}=0$ for all $|\alpha|=r-1$.

Since $\bfT = \lambda\mathbf{I} + \bfN$ is $(A,2)$-isometric, the tuple $\bfN$ is an $A$-root of the polynomial
\[p(\bfx,\bfy) = \Big(\sum_{j=1}^{d}(x_j+\bar{\lambda}_j)(y_j+\lambda_j)-1\Big)^2 = \Big(\sum_{j=1}^{d} x_jy_j + \lambda_jx_j+\bar{\lambda}_jy_j\Big)^2.\] 
On the other hand, $\bfN$ is an $A$-root of $\bfx^{\alpha}$ and $\bfy^{\alpha}$ for all $|\alpha|=r$. This shows that $p(\bfx,\bfy)$, $\bfx^{\alpha}$ and $\bfy^{\alpha}$ belong to $\calJ(A;\bfN,\bfN)$ for all $|\alpha|=r$. To simplify the notation, we shall denote $\calJ(A;\bfN,\bfN)$ by $\calJ$ in the rest of the proof. Take any multiindex $\beta$ with $|\beta|=r-2$. We write
\begin{align*}
\bfx^{\beta}p(\bfx,\bfy)\bfy^{\beta} & = \bfx^{\beta
}\Big(\sum_{j=1}^{d} \lambda_jx_j\Big)\Big(\sum_{\ell=1}^{d}\bar{\lambda}_{\ell}y_{\ell}\Big)\bfy^{\beta} + \sum_{|\gamma|\geq r}x^{\gamma}\,H_{\gamma}(\bfx,\bfy) + G_{\gamma}(\bfx,\bfy)\,\bfy^{\gamma}
\end{align*}
for some polynomials $H_{\gamma}$ and $G_{\gamma}$. Since the left-hand side and the second term on the right-hand side belong to $\calJ$, which is an ideal, we conclude that
\[\bfx^{\beta
}\Big(\sum_{j=1}^{d} \lambda_jx_j\Big)\Big(\sum_{\ell=1}^{d}\bar{\lambda}_{\ell}y_{\ell}\Big)\bfy^{\beta}\in \calJ.\]
Proposition \ref{P:A_positive_hereditary} shows that both $\Big(\sum_{\ell=1}^{d}\bar{\lambda}_{\ell}y_{\ell}\Big)\bfy^{\beta}$ and $\bfx^{\beta
}\Big(\sum_{j=1}^{d} \lambda_jx_j\Big)$ are in $\calJ$. Now for any multiindex $\gamma$ with $|\gamma|=r-3$, we compute
\begin{align*}
\bfx^{\gamma}p(\bfx,\bfy)\bfy^{\gamma} & = \bfx^{\gamma}\big(\sum_{j=1}^{d}x_jy_j\big)^2\bfy^{\gamma} + \sum_{|\beta|=r-2}\bfx^{\beta
}\Big(\sum_{j=1}^{d} \lambda_jx_j\Big)P_{\beta}(\bfx,\bfy)\\
& \qquad\quad + \sum_{|\beta|=r-2}\Big(\sum_{j=1}^{d} \lambda_jx_j\Big)\bfy^{\beta}Q_{\beta}(\bfx,\bfy).
\end{align*}
Since the left-hand side and the last two sums on the right-hand side belong to $\calJ$, it follows that $\bfx^{\gamma}(\sum_{j=1}^{d} x_jy_j)^2\bfy^{\gamma}$ belongs to $\calJ$. Another application of Proposition \ref{P:A_positive_hereditary} then shows that $y_jy_{\ell}\bfy^{\gamma}$ belongs to $\calJ$ for all $1\leq j,\ell\leq d$. That is, $\bfy^{\alpha}$ belongs to $\calJ$ whenever $|\alpha|=r-1$  (as long as $r\geq 3$). As a consequence, we see that $\bfy^{\alpha}$, and hence $\bfx^{\alpha}$, belong to $\calJ$ for all $|\alpha|=2$. This together with the fact that $p(\bfx,\bfy)\in\calJ$ forces $(\sum_{j=1}^{d} \lambda_jx_j)(\sum_{\ell=1}^{d}\bar{\lambda}_{\ell}y_{\ell})$ to belong to $\calJ$, which implies that $\sum_{\ell=1}^{d}\bar{\lambda}_{\ell}y_{\ell}$ is in $\calJ$. We have then shown $AN_jN_{\ell}=0$ for all $1\leq j,\ell\leq d$ and $\sum_{\ell=1}^{d}S_{\ell}^{*}AN_{\ell}=\sum_{\ell=1}^{d}\bar{\lambda}_{\ell}AN_{\ell}=0$, as desired.

Now let us consider $\bfT$ a $2$-isometric tuple on a finite dimensional space $\mcH$. Recall that we have the decomposition $\mcH=\oplus_{\lambda\in\Lambda}\mcH_{\lambda}$ such that for each $\lambda\in\Lambda$, the subspace $\mcH_{\lambda}$ is invariant for $\bfT$ and $T_j-\lambda_j I$ is nilpotent on $\mcH_{\lambda}$. By Proposition \ref{P:generalized_eigenvectors}, we have $(\inner{\omega}{\lambda}-1)^2\inner{v}{u}=0$ for all $v\in\mcH_{\omega}$ and $u\in\mcH_{\lambda}$. It follows that $|\lambda|=1$ for all $\lambda\in\Lambda$ and $\mcH_{\lambda}\perp\mcH_{\omega}$ whenever $\lambda\neq\omega$. As a result, each subspace $\mcH_{\lambda}$ is reducing for $\bfT$. To complete the proof, it suffices to consider $\mcH=\mcH_{\lambda}$. We shall show that either $\bfT$ is a spherical unitary or it is unitarily equivalent to a tuple given in Example \ref{Ex:two_iso_tuples}. Indeed, we have $\bfT=\lambda\bfI + \bfN$, where $\sum_{\ell=1}^d\bar{\lambda}_{\ell}N_{\ell}=0$ and $N_{j}N_{\ell}=0$ for all $1\leq j,\ell\leq d$. If $\bfN=0$, then $\bfT$ is a spherical unitary. Otherwise, let $\mcM=\ker(N_1)\cap\cdots\cap\ker(N_d)$. Then $N_{\ell}(\mcH)\subseteq\mcM$ for all $1\leq \ell\leq d$. As a consequence, with respect to the orthogonal decomposition $\mcH = \mcM\oplus\mcM^{\perp}$, each $N_{\ell}$ has the form
\[
N_{\ell} = \begin{pmatrix}
0 & V_{\ell}\\
0 & 0
\end{pmatrix}
\]
for some $V_{\ell}:\mcM^{\perp}\rightarrow\mcM$. Since $\sum_{\ell=1}^{d}\bar{\lambda}_{\ell}N_{\ell}=0$, we have $\sum_{\ell=1}^{d}\bar{\lambda}_{\ell}V_{\ell}=0$. It follows that $\bfT$ is unitarily equivalent to an operator tuple in Example \ref{Ex:two_iso_tuples}.
\end{proof}

\subsection*{Acknowledgements} The author wishes to thank the referee for helpful suggestions which improved the presentation of the paper.


\def\cprime{$'$}
\providecommand{\bysame}{\leavevmode\hbox to3em{\hrulefill}\thinspace}
\providecommand{\MR}{\relax\ifhmode\unskip\space\fi MR }
\providecommand{\MRhref}[2]{%
  \href{http://www.ams.org/mathscinet-getitem?mr=#1}{#2}
}
\providecommand{\href}[2]{#2}

\end{document}